\documentclass[12pt]{article}

\usepackage[english]{babel}
\usepackage{amssymb}
\usepackage{amsmath}
\usepackage{amsthm}
\usepackage[T1]{fontenc}
\usepackage[all]{xy}

\newtheorem{teo}{Theorem}[section]
\newtheorem{lema}[teo]{Lemma}
\newtheorem{prop}[teo]{Proposition}
\newtheorem{corol}[teo]{Corollary}

\newtheorem{exe}[teo]{Example}
\newtheorem{defini}[teo]{Definition}
\newtheorem*{thmA}{Theorem A}

\usepackage{amssymb,amscd,amsfonts,amsmath,graphicx,amsthm}

\usepackage[top=3cm,bottom=3cm,left=3cm,right=2.5cm]{geometry} 
\begin{document}

\title{Commuting graph of a group on a transversal}

\author{Julio C. M. Pezzott$^{a}$\footnote{The author acknowledges a scholarship from CAPES (Brazil)} , Irene N. Nakaoka\ $^{b}$ \vspace*{3mm}\\
 $^{a,b}$Universidade Estadual de Maring\'{a}, Brazil \\
$^{a}$ e-mail: juliopezzott@gmail.com \\
$^{b}$ e-mail: innakaoka@uem.br \\
 }

\maketitle
\begin{abstract}
Given a finite group $G$ and a subset $X$ of $G$, the commuting graph of $G$ on $X$, denoted by ${\cal C}(G,X)$, is the graph that has $X$ as its vertex set with $x,y\in X$ joined by an edge whenever $x\neq y$ and $xy=yx$. 
Let $T$ be a transversal of the center $Z(G)$ of $G$.  When $G$ is a finite non-abelian group and $X=T\setminus Z(G)$, we denote the graph ${\cal C}(G,X)$ by ${\cal T}(G)$. In this paper, we show that ${\cal T}(G)$ is a connected strongly regular graph if and only if $G$ is isoclinic to an extraspecial $2$-group of order at least $32$. We also characterize the finite non-abelian groups $G$ for which the graph ${\cal T}(G)$ is disconnected strongly regular.
\vspace{2mm}



\end{abstract}

\section{Introduction}

Let $X$ be a subset of a finite group $G$. The {\em commuting graph} of $G$ on the set $X$,  denoted by ${\cal C}(G,X)$, is the graph that has $X$ as its vertex set with $x,y\in X$ joined by an edge whenever $x\neq y$ and $xy=yx$. Commuting graphs were first studied by Brauer and Fowler in \cite{BF55} with $X=G\setminus \{1\}$. Many papers have investigated ${\cal C}(G,X)$ for different choices of  $X$. For example, the works \cite{BBHR,Fischer,SP} investigated the graph ${\cal C}(G,X)$ when $X$ consists of involutions and the papers \cite{akbari,GP2013,GP2014,morgan,VT2010}  considered   $X=G\setminus Z(G)$, where $Z(G)$ denotes the center of $G$. Let us denote the graph ${\cal C}(G,G\setminus Z(G))$  by $\Gamma(G)$ and refer to it simply as commuting graph of $G$.
We  observe that $\Gamma(G)$ highlights the relations of commutativity between non-central elements of the group $G$. However, given $x\in G$, it is clear that any two elements of the coset $xZ(G)$ commute with each other. Furthermore, as observed in \cite{VT2010}, vertices $x$ and $y$ are adjacent in $\Gamma (G)$ if and only if $g$ and $h$ are adjacent for all $g\in xZ(G)$ and $h\in yZ(G)$. 
In this way, we may investigate the relations of commutativity in $G$ by observing only the relations of commutativity that occur between  non-central elements of a transversal of $Z(G)$ in $G$. Thus,  in this work we consider the graph ${\cal C}(G,X)$ when $G$ is a finite non-abelian group and $X=T\setminus Z(G)$, where $T$ is a transversal of $Z(G)$ in $G$. Note that it is a subgraph of $\Gamma (G)$.
It is easy to see that if $T'$ is another transversal  of $Z(G)$ in $G$, then the graphs ${\cal C}(G,T\setminus Z(G))$ and ${\cal C}(G,T'\setminus Z(G))$ are isomorphic. Hence, we will denote the graph ${\cal C}(G,T\setminus Z(G))$ simply by ${\cal T}(G)$ without mentioning the choice of the transversal.  It is worth mentioning that this type of  graph has already been considered in \cite{GP2013,GP2014,VT2010}. Vahidi and Talebi \cite{VT2010} showed that the graphs $\Gamma (G)$ and ${\cal T}(G)$ have the same independence number and diameter. Moreover, if $ \omega({\cal G})$ denotes the clique number of the graph ${\cal G}$, then $\omega(\Gamma (G))=\omega({\cal T}(G))|Z(G)|$. In  \cite{GP2013,GP2014} the authors examined the graph ${\cal T}(G)$ in their study of the diameter of a commuting graph.

We prove that  ${\cal T}(G)$  presents the following relevant property: if $G$ and $H$ are isoclinic groups, then the graphs ${\cal T}(G)$ and ${\cal T} (H)$ are isomorphic (Proposition \ref{isoclinismocomutante}). This helps us to classify groups $G$  such that ${\cal T}(G)$ has certain particular properties.

We recall a graph is regular if each vertex has the same number of neighbors and it is $k$-regular if each vertex has exactly $k$ neighbours. A \emph{strongly regular graph } with parameters $(v,k,\lambda,\mu)$, where $0<k<v-1$, is a graph $k$-regular on $v$ vertices such that   each pair of adjacent vertices has precisely $\lambda$ common neighbours and 
 any two non-adjacent vertices have exactly $\mu$ common neighbours. 

Akbari and Moghaddamfar  \cite[Corollary 2]{akbari} showed that if $\Gamma(G)$ is strongly regular, then it is a disjoint union of at least two complete graphs  and, therefore, it is disconnected. In this work, we prove that there exist groups $G$ for which  the associated graph ${\cal T} (G)$ is   connected strongly regular. Moreover, we present a description of these groups; more precisely, we get:

\begin{thmA}  \label{gfrfinal}
Let $G$ be a finite non-abelian group. Then ${\cal T}(G)$ is a connected strongly regular graph if and only if $G$  admits a decomposition $G=A\times P$, where $A$ is an abelian subgroup of $G$ and $P$ is a Sylow $2$-subgroup of $G$  isoclinic to an extraspecial $2$-group of order $2^{2n+1}$, for some integer $n\geq 2$. In this case, the parameters of ${\cal T}(G)$ are $(2^{2n} - 1, 2^{2n-1} -2, 2^{2n-2}-3, 2^{2n-2}-1)$.
\end{thmA}

We also characterize the finite non-abelian groups $G$ for which the graph ${\cal T}(G)$ is disconnected strongly regular (see Corollary \ref{caracterizacao2}).

\vspace{0.2cm}
In this text, we use the following notations and conventions: $C_n$ is the cyclic group of order $n$ and  $D_{2m}$ (with $m\geq 3$) is the dihedral group of order $2m$. Let $G$ be a group. For $x,y\in G$, the \emph{commutator} of $x$ and $y$ is given by $[x,y]=xyx^{-1}y^{-1}$ and the derived subgroup of $G$ is denoted by $G^{\prime}$. We write $C_G(x)$ for the \emph{centralizer} of $x$ in $G$.  The symbol $cs(G)$ represents the set formed by the sizes of the conjugation classes of the  elements of $G$. For $1<m_1<\ldots< m_n$, we say that $G$ is of \emph{conjugate type} $\{1,m_1,...,m_n\}$ if $cs(G)=\{1,m_1,...,m_n\}$. Thus, given a positive integer $m$, the group  $G$ is of conjugate type $\{1,m\}$ if $[G:C_G(x)]=m$, for any $x\in G\setminus Z(G)$.

Given a graph $\mathcal{G}$ its  vertex set  is represented by $V(\mathcal{G})$ and its edge set by $E(\mathcal{G})$. The neighborhood of a vertex $x$ of $\mathcal{G}$ is denoted by $N(x)$  $(=\{ y \in V(\mathcal{G}) \, :\, \{x,y\} \in E(\mathcal{G})\})$ with degree deg$\,x=|N(x)|$.   When the graphs $\mathcal{G}_1$ and $\mathcal{G}_2$ are isomorphic, we write $\mathcal{G}_1\cong \mathcal{G}_2$ and the disjoint union of the graphs $\mathcal{G}_1$ and $\mathcal{G}_2$ is denoted by $\mathcal{G}_1\cup \mathcal{G}_2$.  As usual, the complete graph on $n$ vertices is denoted by $K_{n}$ and the disjoint union of $r$ complete graphs $K_n$ is indicated by $rK_n$. The graph on $m$ vertices and without edges is represented by $I_m$. Other concepts and basic results on graphs can be seen in \cite{ChartrandZhang}.

\section{Preliminary results}

In this section we present some concepts and results which will be needed to the proof of Theorem A.

First, we observe  the commuting graph $\Gamma (G)$ of $G$ can be obtained from ${\cal T}(G)$  in the following manner: for each vertex $x$ of ${\cal T}(G)$, put $V_x =xZ(G)$. We can see that $\Gamma(G)$ is the graph that has $V:=\bigcup_{x\in V({\cal T}(G))} V_x$ as vertex set and such that a subset $\{a,b\}$ of $V$ is an edge of $\Gamma(G)$ if and only if one of the following conditions is satisfied: (1) $a,b\in V_x$, for some vertex $x$ of ${\cal T}(G)$; or (2) $a\in V_x$ and $b\in V_y$, where $x$ e $y$ are adjacent vertices in ${\cal T}(G)$.

\begin{exe}\label{M3}
\emph{Consider the group $M(3)=\langle a,b,c:a^{3},b^{3},c^{3},[a,b]c^{-1},[a,c],[b,c]\rangle$, which is an extraspecial $3$-group of order $27$. Since $Z(M(3))=\{e,c,c^2\}$, there are eight non-trivial cosets of $Z(M(3))$ in $M(3)$ and $T=\{1,a,a^2,b,b^2,ab,ab^2,a^2b,a^2b^2\}$ is a transversal of $Z(M(3))$ in $M(3)$. The commuting graph of $M(3)$ is isomorphic to $4K_6$ and it can be seen in  Figure \ref{figM3Delta}. Each connected component of $\Gamma(M(3))$ contains exactly two distinct  cosets of $Z(M(3))$. In the figure below,  in each  connected component of $\Gamma(M(3))$, vertices that are in the same coset of $Z(M(3))$ are represented with the same color. In  Figure \ref{figM3}, we see that ${\cal T}(M(3))\cong 4K_2$.}
\end{exe}

\vspace{.05cm}
\begin{figure}[h!]
\centering
\includegraphics[scale=0.45]{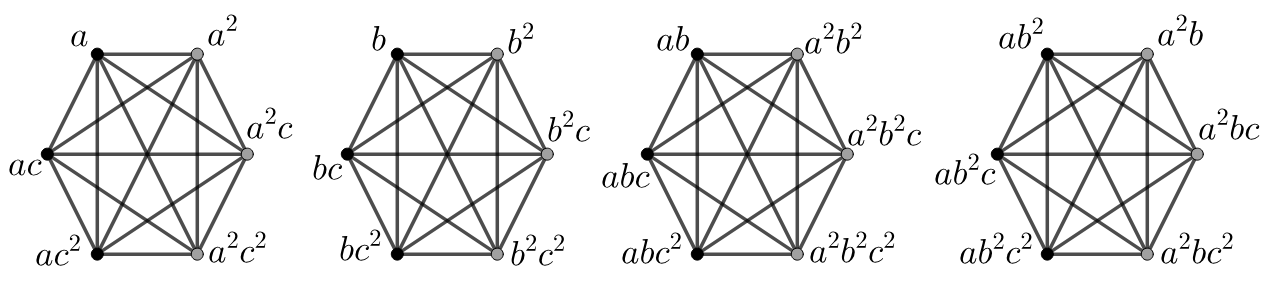}
\caption{$\Gamma(M(3))\cong 4K_6$ }\label{figM3Delta}
\end{figure}

\vspace{.1cm}
\begin{figure}[h!]
\centering
\includegraphics[scale=0.45]{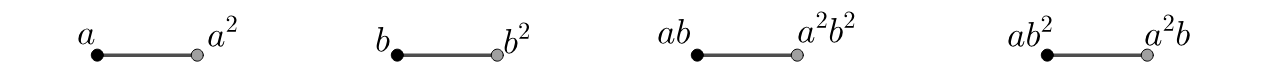}
\caption{${\cal T}(M(3))\cong 4K_2$ }\label{figM3}
\end{figure}

 The  concept below  was introduced by Hall \cite{hall}. We observe that for all group $G$,  the commutator map $\alpha_G:G/Z(G)\times G/Z(G)\rightarrow G'$  given by $\alpha_G(xZ(G),yZ(G))=[x,y]$  is well defined. 

\begin{defini}
\emph{Let $G$ and $H$ be two groups. A pair $(\varphi,\psi)$ is an \emph{isoclinism} from $G$ to $H$ if the following conditions are satisfied:}
\begin{description}
\item[(i)] \emph{ $\varphi$ is an isomorphism from $G/Z(G)$ to $H/Z(H)$;}
\item[(ii)]\emph{$\psi$ is an isomorphism from $G'$ to $H'$;}
\item[(iii)]\emph{  $\psi(\alpha_G(xZ(G),yZ(G)))=\alpha_H(\varphi(xZ(G)),\varphi(yZ(G)))$, for all  $x,y\in G$.} 
\end{description}
\emph{When there is an isoclinism from $G$ to $H$, we  say that the groups $G$ and $H$ are \emph{isoclinic}.}
\end{defini}

The relation of isoclinism is a equivalence relation on groups and the equivalence class of a group $G$ is called of \emph{isoclinism family of $G$}. 

In general, isoclinic groups do not  have isomophic commuting graphs, however for the graphs ${\cal T}(G)$ we have the following:

\begin{prop}\label{isoclinismocomutante}
Let $G$ and $H$ be finite non-abelian groups. If $G$ and $H$ are isoclinic, then the graphs ${\cal T}(G)$ and ${\cal T}(H)$ are isomorphic.
\end{prop}
\begin{proof}
By hypothesis, there are isomorphisms $\varphi:G/Z(G)\rightarrow H/Z(H)$ and $\psi:G'\rightarrow H'$ such that  $\psi(\alpha_G(xZ(G),yZ(G)))=\alpha_H(\varphi(xZ(G)),\varphi(yZ(G)))$, for all $x,y\in G$. Let $T(G)=\{1,t_1,\ldots,t_n\}$ be a transversal of $Z(G)$ in $G$, and, for each $i\in\{1,\ldots,n\}$, let $y_i$ be a representative  of the coset $\varphi(t_iZ(G))$. Therefore $T(H)=\{1,y_1,\ldots,y_n\}$ is a transversal of $Z(H)$ in $H$. If we put $T(G)\setminus\{1\}$ as the 
vertex set of ${\cal T}(G)$ and $T(H)\setminus\{1\}$ as the vertex set of ${\cal T}(H)$, the mapping $t_i\longmapsto y_i$ defines a bijection between $V({\cal T}(G))$ and $V({\cal T}(H))$. In addition, if $t_i$ e $t_j$ are adjacent in ${\cal T}(G)$, then $[y_i,y_j]=\psi([t_i,t_j])=1$, that is, $y_i$ and $y_j$ are adjacent in ${\cal T}(H)$. On the other hand, if $[y_i,y_j]=1$, we have $\psi([t_i,t_j])=1$. Since $\psi$ is an isomorphism, we obtain $[t_i,t_j]=1$. Therefore ${\cal T}(G)$ and ${\cal T}(H)$ are isomorphic.
\end{proof}

The next example shows  that the converse of  Proposition
\ref{isoclinismocomutante} does not hold in general.

\begin{exe}\label{contraexemploisoclinismocomutante}
\emph{Consider the following groups :}

\emph{$J:= \langle a,b,c \mid a^4 = b^4 = c^4 = 1, a^2 = b^2, ab = ba, cac^{-1} = a^{-1}, bcb^{-1} = c^{-1} \rangle$}

\emph{$D_{16}:=\langle x,y\mid x^8=y^2=1, \, yxy^{-1}=x^{-1}\rangle$.}

\emph{We have  $Z(J)=\{1,a^2,c^2,a^2c^2\}$ and $Z(D_{16})=\{1,x^4\}$; further $T_1=\{1,a,b,c,ab,bc,ac,abc\}$ is a transversal of $Z(J)$ in $J$ and $T_2=\{1,x,x^2,x^3,y,xy,x^2y,x^3y\}$ is a transversal of $Z(D_{16})$ in $D_{16}$. It is not difficult verify that ${\cal T}(J)\cong {\cal T}(D_{16})\cong K_3\cup I_4$ (Figures \ref{grupoJ} and \ref{D16}); however, $J$ e $D_{16}$ are not isoclinic, because the groups $J/Z(J)$,  $D_{16}/Z(D_{16})$ are isomorphic to $C_2^3$, $D_8$, respectively. }
\end{exe}

\begin{figure}[htb]
\centering
\begin{minipage}[c]{0.5\linewidth}
\centering
\includegraphics[scale=0.45]{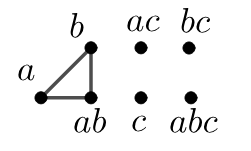}
\caption{${\cal T}(J)$ }\label{grupoJ}
\end{minipage}\hfill
\begin{minipage}[c]{0.5\linewidth}
\centering
\includegraphics[scale=0.45]{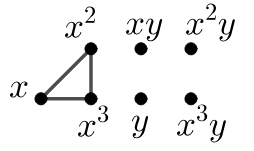}
\caption{${\cal T}(D_{16})$}\label{D16}
\end{minipage}\hfill
\end{figure}

The next result gives the size of the common neighborhood of two vertices of  ${\cal T}(G)$.

\begin{prop}\label{grauX}
Given a finite non-abelian group $G$, let $x$ and $y$ be distinct vertices of ${\cal T}(G)$.  The following statements are true:
\begin{description}
\item[(i)] deg$(x)=[C_G(x):Z(G)] - 2$;
\item[(ii)] if $x$ and $y$ are adjacent, then $|N(x)\cap N(y)|=[C_G(x)\cap C_G(y):Z(G)]-3$;
\item[(iii)] if $x$ and $y$ are not adjacent, then $|N(x)\cap N(y)|=[C_G(x)\cap C_G(y):Z(G)]-1$.
\end{description} 
\end{prop}
\begin{proof}
Let $x$ be a vertex of ${\cal T}(G)$. If $N(x)=\{x_1,\ldots,x_n\}$ is the neighborhood of $x$ in ${\cal T}(G)$ and if $\mathcal{A}:=\bigcup_{i=1}^n x_iZ(G)$, it is easy to verify that $C_G(x)=Z(G)\cup xZ(G)\cup \mathcal{A}$, so that  deg$(x)=n=[C_G(x):Z(G)] - 2$. This proves (i).

Let $x$ and $y$ be distinct vertices of ${\cal T}(G)$ and put $N=N(x)\cap N(y)$ and $\mathcal{B}:=\bigcup_{a\in N} aZ(G)$. We note that $\{x,y\}\cap  N=\emptyset$. If $x$ and $y$ are adjacent in ${\cal T}(G)$, then $\{x,y\}\subseteq C_G(x)\cap C_G(y)$ and, in this case, we obtain $C_G(x)\cap C_G(y)=Z(G)\cup xZ(G)\cup yZ(G)\cup \mathcal{B}$. It follows that $|N|=[C_G(x)\cap C_G(y):Z(G)]-3$, which proves (ii). If $x$ and $y$ are not adjacent in ${\cal T}(G)$, then $x\notin C_G(y)$ and $y\notin C_G(x)$ and it is easy to see that $C_G(x)\cap C_G(y)=Z(G)\cup \mathcal{B}$. Therefore $|N|=[C_G(x)\cap C_G(y):Z(G)]-1$ and  part (iii) is proved. 
\end{proof}

From Proposition \ref{grauX} it follows that  ${\cal T}(G)$ is regular if and only if  $|C_G(x)|=|C_G(y)|$, for all $x,y \in G \setminus Z(G)$, that is,  $|cs(G)|=2$. Ito \cite{ito} gave a description of the finite groups of conjugate type $\{1,m\}$.

\begin{teo} \label{scGgammaG} \emph{(\cite[Theorem 1]{ito})}
If $G$ is a finite  group of conjugate type $\{1,m\}$, then $m$ is a power of a prime $p$  and $G$ admits a decomposition $G=A\times P$, where $A$ is an abelian subgroup of $G$  and $P$ is a Sylow $p$-subgroup  of the same conjugate type as $G$. 
\end{teo}

Let $p$ be a prime. 
 Recall that a finite non-abelian $p$-group is extraspecial if $G^{\prime} =Z(G)=\Phi(G)$ and $|Z(G)|=p$, where $\Phi(G)$ denotes the Frattini subgroup of $G$. It is well known that any non-abelian group of  order $p^3$ is extraspecial. Further, an extraspecial  $p$-group has order $p^{2n+1}$, for some positive integer $n$.

In this paper we will use the following  Ishikawa's classification: 

\begin{teo}\label{ishik1}
\emph{(\cite[Proposition 3.1]{ishikawa})} Let $G$ be a finite $p$-group, $p$ prime. Then $G$ is of conjugate type $\{1,p\}$ if and only if $G$ is isoclinic to an extraspecial $p$-group.
\end{teo}

\section{Main results}

In this section we will present the proof of  Theorem A,
which will be done in steps. First
let us show that the graph ${\cal T}(G)$ associated to an extraspecial $2$-group of order $2^{2n+1}$, $n\geq 2$, is  connected strongly regular. For it, we will need the following two lemmas:

\begin{lema}\label{lemaextraespecial}
Let $G$ be an extraspecial $2$-group. If $x,y\in G$ and $x\notin yZ(G)$, then $C_G(x)\neq C_G(y)$.
\end{lema}
\begin{proof} 
If $G=D_8$ or $G$ is the quaternion group of order $8$, the result is clearly true. Consider  an extraspecial $2$-group $G$ of order $m$ and suppose that the result is true for all extraspecial $2$-group of order less than $m$. By \cite[5.3.8]{robinson}, $G$ contains subgroups $A$ and $B$ such that $A$ and $B$ are extraspecial $2$-groups of order less than $m$ and $G=AB$ is a central product. Hence,   the elements of $A$ commute with the elements of $B$ and $A\cap B=Z(G)=Z(A)=Z(B)$. By induction hypothesis, the following two conditions are satisfied:

(i) $C_A(a_1)\neq C_A(a_2)$, for all $a_1,a_2\in A$ such that $a_1\notin a_2Z(A)$.

(ii) $C_B(b_1)\neq C_B(b_2)$, for all $b_1,b_2\in B$ such that $b_1\notin b_2Z(B)$.

Let $x=a_1b_1$ and $y=a_2b_2$ be non-central elements of $G$ such that $x\notin yZ(G)$, where $a_1,a_2\in A$ and $b_1,b_2\in B$. Thus, in virtue of (i) and (ii), we must have $C_A(a_1)\neq C_A(a_2)$ or $C_B(b_1)\neq C_B(b_2)$. In the first situation, there is $w\in C_A(a_1)\setminus C_A(a_2)$ (or there is $u\in C_A(a_2)\setminus C_A(a_1)$) and it follows that $w\in C_G(x)\setminus C_G(y)$ (or $u\in C_G(y)\setminus C_G(x)$), which gives us $C_G(x)\neq C_G(y)$. Analogously, $C_B(b_1)\neq C_B(b_2)$ implies $C_G(x)\neq C_G(y)$, as required.
\end{proof}

 The following result is taken from \cite{godsil}.

\begin{lema}\label{completodesconexo}\emph{ (\cite[Lemma 10.1.1]{godsil}
)}If ${\cal G}$ is a strongly regular graph with parameters $(v,k,\lambda, \mu)$, then the following statements are equivalent:
\begin{description}
\item[(i)] ${\cal G}$ is disconnected;
\item[(ii)] $\mu=0$;
\item[(iii)] $\lambda= k-1$;
\item[(iv)] ${\cal G}$ is isomorphic to $mK_{k+1}$ for some integer $m>1$.
\end{description} 
\end{lema}

\begin{prop}\label{extraespecialGFR}
If $G$ is an extraspecial $2$-group of order $2^{2n+1}$, with $n\geq 2$, then ${\cal T}(G)$ is a connected strongly regular graph with parameters $(2^{2n} - 1, 2^{2n-1} -2, 2^{2n-2}-3, 2^{2n-2}-1)$.
\end{prop}
\begin{proof}
Let $x,y\in G\setminus Z(G)$ be such that $x\notin yZ(G)$. By Lemma \ref{lemaextraespecial}, we have $C_G(x)\neq C_G(y)$ and, by Theorem \ref{ishik1}, $G$ is of conjugate type $\{1,2\}$. This shows that $[C_G(x):C_G(x)\cap C_G(y)]\geq 2$. Write $G'=\{1,d\}$. Given $g,h\in C_G(x)\setminus C_G(y)$, we have $gy=ygd$,  $h^{-1}y=yh^{-1}d$ and so $(gh^{-1})y=y(gh^{-1})$, that is, $gh^{-1}\in C_G(x)\cap C_G(y)$. Hence $[C_G(x):C_G(x)\cap C_G(y)]= 2$. Applying Proposition \ref{grauX} to calculate the degree of each vertex of ${\cal T}(G)$ and the size of the common neighborhood of two vertices, we conclude that ${\cal T}(G)$ is   strongly regular with parameters $(2^{2n} - 1, 2^{2n-1} -2, 2^{2n-2}-3, 2^{2n-2}-1)$. Since $2^{2n-2}-1 \neq 0$, Lemma \ref{completodesconexo} give us ${\cal T}(G)$ is connected.
\end{proof}

The next result establishes  a necessary condition for a graph ${\cal T}(G)$   to be connected strongly regular.

\begin{lema}\label{condicaonecessariaconexo}
Let $G$ be a finite non-abelian group such that ${\cal T}(G)$ is  connected strongly regular. Then, for all $x,y\in G\setminus Z(G)$ such that $x\notin yZ(G)$, we have $C_G(x)\neq C_G(y)$.
\end{lema}
\begin{proof}
Suppose there is a non-abelian group $G$ such that ${\cal T}(G)$ is a connected strongly regular graph with parameters $(v,k,\lambda,\mu)$ and that there are $x,y\in G\setminus Z(G)$ such that $x\notin yZ(G)$ and $C_G(x)=C_G(y)$. For all $g\in xZ(G)$ and $h\in yZ(G)$ we have $[g,h]=[x,y]=1,$ $C_G(g)=C_G(x)$ and  $C_G(h)=C_G(y)$. Hence, we may assume that $x,y$ are vertices of ${\cal T}(G)$. From  Proposition \ref{grauX}, it follows that if  $n=[G:Z(G)]$ and $m=[C_G(x):Z(G)]$, then $v=n-1$, $k=m -2$ and $\lambda=m - 3=k-1$. By Lemma \ref{completodesconexo} the graph ${\cal T}(G)$ is  disconnected, a contradiction. 
\end{proof}

Our intention is to characterize the groups $G$ for which ${\cal T}(G)$  is connected strongly regular . As we   deal with a case where the graph  ${\cal T}(G)$ is regular (and so $G$ is of conjugate type $\{1,m\}$), in virtue of Theorem \ref{scGgammaG}, we can restrict our study to $p$-groups. 

\begin{lema}\label{gfrpimpar}
Given an odd prime $p$ let $G$ be a finite non-abelian $p$-group. If ${\cal T}(G)$ is strongly regular, then ${\cal T}(G)$ is disconnected.
\end{lema}
\begin{proof}
Suppose that ${\cal T}(G)$ is strongly regular. Given $x\in G\setminus Z(G)$, we have  $x^2\notin xZ(G)$; further, as $C_G(x)\subseteq C_G(x^2)$ and deg$(x)$=deg$(x^2)$ we get  $C_G(x)= C_G(x^2)$. Hence, by Lemma \ref{condicaonecessariaconexo},  the graph ${\cal T}(G)$ is disconnected, as desired.
\end{proof}

 Lemma \ref{gfrpimpar} say us that if $G$ is a $p$-group such that ${\cal T}(G)$ is a connected strongly regular graph, then $p=2$. 

For the proof of the next two results we will need the following relation, that  holds for any  strongly regular graph with parameters $(v,k,\lambda,\mu)$ (see  \cite[Proposition 2.6]{cameron}): 
\begin{equation}\label{relacaoparametrosGFR}
\mu(v-k-1)=k(k-\lambda - 1).
\end{equation}

\begin{lema}\label{srgconexo1}
Let $G$ be a finite non-abelian $2$-group such that ${\cal T}(G)$ is a connected strongly regular graph with parameters $(v,k,\lambda, \mu)$. Suppose that $|C_G(x)\cap C_G(y)|=|C_G(t)\cap C_G(w)|$, for any vertices $x,y,t,w$ of ${\cal T}(G)$, with $x\neq y$ and $t\neq w$. Then $G$ is isoclinic to an extraspecial $2$-group of order at least $32$.
\end{lema}
\begin{proof} From the hypotheses it follows that
there exist non-negative integers $m,n,r$ such that  $[G:Z(G)]=2^{m+n+r}$, $[C_G(x):Z(G)]=2^{m+n}$ and $[C_G(x)\cap C_G(y): Z(G)]=2^{m}$, for any distinct vertices $x$ and $y$ of ${\cal T}(G)$. Thus, Proposition \ref{grauX} provides us
\[v=2^{m+n+r}-1, \qquad k=2^{m+n}-2, \qquad \lambda=2^{m}-3, \qquad \mu=2^{m}-1. 
\]
Because $\lambda \geq 0$, we must have $m\geq 2$. 
From  (\ref{relacaoparametrosGFR}) we obtain $(2^{n+r} - 2^{n})(2^{m}-1)=(2^{n}-1)(2^{m+n} - 2)$ and, consequently, 
$2^n(2^{m+r} - 2^r  + 1 - 2^{m+n} + 2)=2.$
Because $\mu\leq k$ we get $n\geq 1$ and from the last equality above we conclude  $n=1$. Thus $2^{m+r} - 2^r  + 1 - 2^{m+1} + 2=1$ and so $2^m(2^r-2)=2^r - 2$, which implies $r=1$. Hence, $[G:C_G(x)]=2$, for any $x\in G\setminus Z(G)$, that is, $G$ is of conjugate type $\{1,2\}$. Now, from  Theorem \ref{ishik1} it follows that $G$ is isoclinic to an extraspecial $2$-group $E$. As $m\geq 2$ and $n=r=1$ we obtain $[E:Z(E)]=[G:Z(G)]\geq 16$ and so $E$ has order at least $32$.
\end{proof}

\begin{teo}\label{srgconexo2}
If $G$ is a finite non-abelian $2$-group such that ${\cal T}(G)$ is  connected strongly regular, then $G$ is isoclinic to an extraspecial $2$-group of order at least $32$.
\end{teo}
\begin{proof}
By Lemma \ref{srgconexo1}, it is sufficient to show that $|C_G(x)\cap C_G(y)|=|C_G(t)\cap C_G(w)|$, for any vertices $x,y,t,w$ of ${\cal T}(G)$, where $x\neq y$ and $t\neq w$. Let us assume the contrary, that is, there are vertices $x,y,t,w$ of ${\cal T}(G)$ such that 
\begin{equation} \label{desigualdade1}
|C_G(x)\cap C_G(y)|>|C_G(t)\cap C_G(w)|. 
\end{equation}
Taking into account  Proposition \ref{grauX} and the fact that   ${\cal T}(G)$ is  strongly regular, we have two cases to analyze:

\emph{{\underline{Case 1}}}:  $[x,y]\neq 1$ and $[t,w]= 1$. Due to Proposition \ref{grauX} and (\ref{desigualdade1}), we can write
$$v=2^{m+n+r+s}-1 \qquad k=2^{m+n+r}-2 \qquad \lambda=2^{m}-3, \qquad \mu=2^{m+n}-1,$$
where $m,n,r,s$ are non-negative integers, with $n\geq 1$.
From $\mu(v-k-1)=k(k-\lambda - 1)$ we obtain 
$$(2^{m+n} - 1)(2^{n+r+s}-2^{n+r})=(2^{m+n+r}-2)(2^{n+r} - 1),$$
 which is equivalent to
$$2^n(2^{m+n+r+s} - 2^{m+n+r}-2^{r+s}+2^r - 2^{m+n+2r} + 2^{m+r} + 2^{r+1})=2.$$
Since $n\geq 1$, this forces $n=1$ and 
$$2^{m+n+r+s} - 2^{m+n+r}-2^{r+s}+2^r - 2^{m+n+2r} + 2^{m+r} + 2^{r+1}=1,$$
that is, $2^r(2^{m+s+1} - 2^{m+1} - 2^s + 1 - 2^{m+r+1} + 2^m +2)=1$. For such equality to occur, we must have $r=0$, which give us $k<\mu$, a contradiction.

\emph{{\underline{Case 2}}}: $[x,y]= 1$ and $[t,w]\neq 1$. Considering (\ref{desigualdade1}), we may write
$$v=2^{m+n+r+s}-1, \qquad k=2^{m+n+r}-2, \qquad \lambda=2^{m+n}-3, \qquad \mu=2^{m}-1,$$
where $m,n,r,s$ are non-negative integers, with $n\geq 1$. Note that $s,m\neq 0$ because $k\neq v-1$ and $\mu \neq 0$.
From the identity $\mu(v-k-1)=k(k-\lambda - 1)$ we get $(2^{m} - 1)(2^{r+s}-2^{r})=(2^{m+n+r}-2)(2^{r} - 1)$ and so $2^r(2^{m+s} - 2^m - 2^s + 1-2^{m+n+r} + 2^{m+n} +2)=2$, which implies $r\in\{0,1\}$. However, by Lemma \ref{completodesconexo}, $\lambda \neq  k-1$; thus $r=1$ and, consequently, $$2^{m+s} - 2^m - 2^s + 1-2^{m+n+1} + 2^{m+n} +2=1.$$
It follows from the last equality that $2^{m}(2^s - 1 - 2^{n+1} +2^n)=2(2^{s-1} - 1)$. 

If $s=1$, we obtain $1-2^{n+1}+2^n=0$, that is, $n=0$, a contradiction.

Suppose now $s\geq 2$. Since $m\geq 1$, we conclude that the number $2^{m-1}(2^s - 1 - 2^{n+1} +2^n)=2^{s-1} - 1$ is odd and this implies $m=1$. From the penultimate equality we get $2^s - 1 - 2^{n+1}+2^n = 2^{s-1} - 1$, which gives us $2^{s-1}=2^n$, that is, $s-1=n$. Hence, in this case, the parameters of ${\cal T}(G)$ are:
\begin{equation} \label{parametros1}
v=2^{2s+1}-1, \qquad k=2^{s+1}-2, \qquad \lambda=2^{s}-3, \qquad \mu=1.
\end{equation}
Putting $\gamma=(\mu-\lambda)^2 + 4(k-\mu)$, it follows from (\cite{cameron}, Theorem 2.16) that the numbers $m_1$ and $m_2$ given below are non-negative integers:
$$\displaystyle{m_1=\frac{1}{2}\left(v-1 +\frac{2k+(v-1)(\mu-\lambda)}{\sqrt{\gamma}}\right)}, \qquad \displaystyle{m_2=\frac{1}{2}\left(v-1 -\frac{2k+(v-1)(\mu-\lambda)}{\sqrt{\gamma}}\right)}.$$
Note that  $2k+(v-1)(\mu-\lambda)=2^{s+1}(2^{s+2}-2^{2s}+3)-12$ and this number is nonzero  for $s\geq 2$.

From the parameters obtained in (\ref{parametros1}), we obtain $\gamma=(4-2^s)^2 + 4(2^{s+1} - 3)=4(4^{s-1} + 1)$ and, thus, $\sqrt{\gamma}=2\sqrt{4^{s-1}+1}$. Since $m_1$ and $m_2$ are non-negative integers, there is $a\in\mathbb{Q}$ such that  $4^{s-1}=a^2 -1$. Then $a$ must be an odd integer, say $a=2l+1$, and  we get $4^{s-1}=4l(l+1)$, a contradiction. 

Now, the result follows from Lemma \ref{srgconexo1}. 
\end{proof}

We are now ready to show Theorem A.

\begin{proof}[{\bf Proof of Theorem A}]
Let $G$ be a finite non-abelian group such that ${\cal T}(G)$ is  connected strongly regular.  Theorem \ref{scGgammaG} tells us that $G$ admits a decomposition $G=A\times P$, where $A$ is an abelian subgroup of $G$ and $P$ is a Sylow $p$-subgroup of $G$ of conjugate type $\{1,p^r\}$ for some positive integer $r$. Now \cite[page 286]{berkovich} ensures us that $G$ and $P$ are isoclinic and so ${\cal T}(G) \cong {\cal T}(P)$.  Lemma  \ref{gfrpimpar} forces $p=2$ and  Theorem \ref{srgconexo2} yields us that  $P$ is isoclinic to an extraspecial $2$-group $E$ of order at least $32$, that is,  $|E|=2^{2n+1}$, for some $n \geq 2$. Now, the result follows from Propositions \ref{extraespecialGFR} and \ref{isoclinismocomutante}, taking into account that
 $G$ and $E$ are isoclinic groups. 
\end{proof}

As an immediate consequence of the proof of Theorem A, we obtain

\begin{corol}
Let $G$ be a finite non-abelian group. The graph ${\cal T}(G)$ is  connected strongly regular if and only if $G$ is isoclinic to an extraspecial $2$-group of order at least $32$.
\end{corol}

We also characterize the finite non-abelian groups $G$ for which the graph ${\cal T}(G)$ is a disjoint union of complete graphs of the same size.

\begin{prop} \label{caracterizacao}
Let $G$ be a finite non-abelian group and let $m,n$ be  positive integers.  The graph ${\cal T}(G)$ is a disjoint union of $m$ copies of $K_{n}$  if and only if   $mn=[G:Z(G)]-1$ and for any $x\in G\setminus Z(G)$ we have $C_G(x)$ is abelian and $[C_G(x):Z(G)]=n+1$.
\end{prop} 
\begin{proof}
Suppose that ${\cal T}(G)$ is a disjoint union of  $m$ copies of $K_{n}$. It is clear that $mn=[G:Z(G)]-1$. Let $T=\{1,x_1,\ldots, x_{mn}\}$ be the transversal of $Z(G)$ in $G$ such that $V({\cal T}(G))=T\setminus\{1\}$. Given $x\in G\setminus Z(G)$, there is $i\in\{1,\ldots, mn\}$ such that $x\in x_iZ(G)$. If $N(x_i)=\{y_1,\ldots,y_{n-1}\}$ is the neighborhood of $x_i$ in ${\cal T}(G)$, then deg$(x_i) = n-1$. Since deg$(x_i) = [C_G(x_i):Z(G)]-2$ (by Proposition \ref{grauX}), we have  $[C_G(x_i):Z(G)]=n+1$. Moreover, $C_G(x_i)=Z(G)\cup x_iZ(G)\cup \left(\bigcup_{y\in N(x_i)} yZ(G)\right)$ and $C_G(x_i)$ is abelian, because the subgraph generated by $\{x_i\}\cup N(x_i)$ is complete. Since $C_G(x_i)=C_G(x)$, we conclude that $[C_G(x):Z(G)]=n+1$ and $C_G(x)$ is abelian. 

Conversely,  assume that  $[G:Z(G)]-1=mn$ and for any $x\in G\setminus Z(G)$, $C_G(x)$ is abelian and $[C_G(x):Z(G)]=n+1$.  Let $x$ be a vertex of ${\cal T}(G)$. If $N(x)=\{y_1,\ldots,y_{n-1}\}$ is the neighborhood of $x$ in ${\cal T}(G)$, then the subgraph generated by $\{x\}\cup N(x)$ is isomorphic to $K_{n}$, because $C_G(x)=Z(G)\cup xZ(G)\cup \left(\bigcup_{y\in N(x)} yZ(G)\right)$ and $C_G(x)$ is abelian. In addiction, as the non-central elements of $G$ have centralizers of equal size, ${\cal T}(G)$ is a disjoint union of $m$  graphs isomorphic to  $K_n$.
\end{proof}

It is worth mentioning that when $G$ is a non-abelian $p$-group, the case $n=1$ in Proposition \ref{caracterizacao} is also a consequence of \cite[Corollary 1]{akbari}. Further, if $G$ is a group such that ${\cal T}(G)\cong mK_n$, the result above produces $m\geq 3$.

We have already seen (\cite[Lemma 10.1.1]{godsil}) that if a graph ${\cal G}$ is disconnected strongly regular, then it is isomorphic to $mK_n$ for some integers $m,n\geq 2$. Of course, the reciprocal also holds.   Thus, as an immediate consequence from Proposition \ref{caracterizacao}, we obtain 

\begin{corol} \label{caracterizacao2}
Let $G$ be a finite non-abelian group. The graph ${\cal T}(G)$ is disconnected strongly regular if and only if the centralizers of the non-central elements of $G$  are all abelian and there exists an integer $r\geq 3$ such that $[C_G(x):Z(G)]=r$, for all $x\in G\setminus Z(G)$.
\end{corol}

Akbari and Moghaddamfar \cite[Corollary 2(a)]{akbari} proved that  a commuting graph $\Gamma (G)$ is strongly regular if and only if it is a disjoint union of complete graphs on $|Z(G)|$ vertices. However, Example \ref{M3} shows that the connected components of a strongly regular graph $\Gamma (G)$ do not necessarily  have size equal to $|Z(G)|$. In fact, we have the following result, which corrects  the error in \cite[Corollary 2(a)]{akbari}:

\begin{prop} \label{correcao}
Let $G$ be a finite non-abelian group. The graph $\Gamma(G)$ is  strongly regular  if and only if $\Gamma(G)\cong mK_s$, where $ms=|G\setminus Z(G)|$ and  $s=([C_G(x):Z(G)]-1)|Z(G)|$, for all $x\in G\setminus Z(G)$.
\end{prop}
\begin{proof}
Suppose that $\Gamma(G)$ is a strongly regular graph with parameters  $(v,k,\lambda,\mu)$. As $\Gamma(G)$ is regular, from Theorem \ref{scGgammaG} it follows $|Z(G)|>1$. For all $x\in G\setminus Z(G)$ and $z_1,z_2\in Z(G)$, with $z_1\neq z_2$, it is clear that $C_G(xz_1)=C_G(xz_2)$ and $[xz_1, xz_2]=1$. Thus, we must have $\lambda=k-1$. As the graph $\Gamma(G)$  is not complete, from the relation  $\mu(v-k-1)=k(k-\lambda - 1)$, we obtain  $\mu=0$. Hence, Lemma \ref{completodesconexo} ensures us $\Gamma(G) \cong mK_s$, for some integers $m,s \geq 2$. Clearly, $ms=|G\setminus Z(G)|$ and by calculating the degree of each vertex we get  $s=|C_G(x)| - |Z(G)|=([C_G(x):Z(G)]-1)|Z(G)|$, for all $x\in G\setminus Z(G)$.
\end{proof}

Let $G$ be a  finite non-abelian group.  Since for $x,y \in G$ we have  $[x,y]=1$ if and only if $[g,h]=1$, for all $g\in xZ(G)$, $h\in yZ(G)$,  as a consequence of  Proposition  \ref{correcao} we conclude that if $\Gamma (G)$ is a strongly regular graph, then each connected component of this graph is an union of $n$ non-trivial cosets of $Z(G)$, where $n=[C_G(x):Z(G)]-1$.

Let $p$ be a prime and let  $P$ be an extraspecial $p$-group. If $p$ is odd and $|P|>p^3$,  Theorem \ref{ishik1} and Lemma \ref{gfrpimpar} ensure us that ${\cal T}(P)$ is regular but not strongly regular. 
The next result tells us that  if $G$ is a  group isoclinic to an extraspecial $p$-group of order $p^3$, then ${\cal T}(G)$ is disconnected strongly regular. 

\begin{prop}\label{extraespecialordemcubo2}
Let $G$ be a finite non-abelian group and let $p$ be a prime number. We have  ${\cal T}(G)\cong (p+1)K_{p-1}$ if and only if $G$ is isoclinic to an extraspecial $p$-group of order $p^3$.
\end{prop}
\begin{proof}
Let $G$ be a group such that ${\cal T}(G)\cong (p+1)K_{p-1}$. Hence $[G:Z(G)]=p^2$ and as the graph ${\cal T}(G)$ is regular, by Lemma \ref{scGgammaG} and Proposition \ref{isoclinismocomutante}, we may assume that $G$ is a $p$-group. Now Proposition \ref{caracterizacao} provides us  $[C_G(x):Z(G)]=p$, for all $x\in G\setminus Z(G)$. Thus we conclude $[G:C_G(x)]=p$, for all $x\in G\setminus Z(G)$ and, consequently,  $G$ is of conjugate type $\{1,p\}$. From Theorem \ref{ishik1} it follows that $G$ is isoclinic to an extraspecial $p$-group, say $P$. Since $[P:Z(P)]=[G:Z(G)]=p^2$ and $|Z(P)|=p$, we must have $|P|=p^3$. The converse follows from Proposition \ref{caracterizacao} and Proposition \ref{isoclinismocomutante}.
\end{proof}

\end{document}